\documentclass[a4paper,11pt]{amsproc}
\usepackage{tikz-cd}
\usepackage{tikz}
\usepackage{amsfonts, amsmath, amssymb, graphicx, mathrsfs}
\usepackage[all]{xy}
\newdir{ >}{!/6pt/@{ }*:(1,0.2)@_{>}*:(1,-0.2)@^{>}}

\theoremstyle{plain}

\newtheorem{theorem}{Theorem}[section]
\newtheorem{lemma}[theorem]{Lemma}
\newtheorem{proposition}[theorem]{Proposition}
\theoremstyle{definition}
\newtheorem{definition}[theorem]{Definition}

\newtheorem{properties}[theorem]{Properties}

\newtheorem{counter example}[theorem]{Counter Example}

\newtheorem{corollary}[theorem]{Corollary}
\newtheorem{example}[theorem]{Example}

\numberwithin{equation}{section}

\keywords{ $(\mathcal{Z}_c)_F$-filters, $(\mathcal{Z}_c)_F$-ideals, structure spaces, Baer-ring, $F_cP$-space, Zero divisor graph\\ * The Corresponding Author, Email: dmandal.cu@gmail.com}


\begin{document}
	
	\title[On the ring $C_c(X)_F$]{Rings of functions which are discontinuous on a finite set with countable range}

	

	\thanks {}
	\maketitle	
	\author{Achintya Singha$^1$, D. Mandal$^{2*}$, Samir Ch Mandal$^3$  and Sagarmoy Bag$^4$}\\
	
	\address{$^{1, 4}$ Department of Mathematics, Bangabasi Evening College, 19 Rajkumar Chakraborty Sarani, Kolkata 700009, West Bengal, India}\\
	\email{Email$^1$: achintya.puremath@gmail.com}\\
		\email{Email$^4$: sagarmoy.bag01@gmail.com}\\

	\address{$^{2,3}$ Department of Pure Mathematics, University of Calcutta, 35, Ballygunge Circular Road, Kolkata 700019, West Bengal, India}\\
	\email{Email$^2$ : dmandal.cu@gmail.com}\\
	\email{Email$^3$: samirchmandal@gmail.com}

	\begin{abstract}
		Consider the ring $C_c(X)_F$ of real valued functions which are discontinuous on a finite set with countable range. We discuss $(\mathcal{Z}_c)_F$-filters on $X$ and $(\mathcal{Z}_c)_F$-ideals of $C_c(X)_F$. We establish an analogous version of Gelfand-Kolmogoroff theorem in our setting. We prove some equivalent conditions when $C_c(X)_F$ is a Baer-ring and a regular ring. Lastly, we talk about the zero divisor graph on $C_c(X)_F$.
		
	\end{abstract}
	
	\section{Introduction}
	We start with a $T_1$ topological space $(X,\tau)$. Let $C_c(X)_F$ be the collection of all real valued functions on $X$ which are discontinuous on a finite set with countable range. Then $C_c(X)_F$ is a commutative ring with unity, where addition and multiplications are defined pointwise. We define for $f,g\in C_c(X)_F$, $f\leq g$ if and only if $f(x)\leq g(x)$ for all $x\in X$. Then $f\vee g=\frac{f+g+\vert f-g\vert}{2}\in C_c(X)_F$ and $f\wedge g=-(-f\vee -g)\in C_c(X)_F$. Thus $(C_c(X)_F, +, \cdot, \leq)$ is a lattice ordered ring. Clearly, $C_c(X)_F\subseteq C(X)_F$ ($\equiv$ rings of functions which are discontinuous on a finite set, studied briefly in \cite{MZ(2021), ZMA(2018)}). It is interesting to see that taking an infinite set $X$ with co-finite topology (or any irreducible topological space),  we get $C_c(X)_F=C(X)_F$. Let $C_c(X)$ be the ring of all real valued continuous functions with countable range. We see that the ring $C_c(X)_F$ properly contains the ring $C_c(X)$. Our intention in this paper is to study some ring properties of $C_c(X)_F$ and interpret certain topological behaviours of $X$ via $C_c(X)_F$ . 
	
	In Section 2, we define $\mathcal{F}_c$-completely separated subsets of $X$ and  establish that two subsets are $\mathcal{F}_c$-completely separated if and only if they are contained in two disjoint zero sets. In Section 3, we develop a connection between $(\mathcal{Z}_c)_F$-filters and ideals of $C_c(X)_F$. We define $(\mathcal{Z}_c)_F$-ideal and some equivalent conditions of $(\mathcal{Z}_c)_F$-prime ideals of $C_c(X)_F$ (see Theorem \ref{SMP3,3.7}) and arrive at a decision that $C_c(X)_F$ is a Gelfand ring. Also, we develop conditions when every ideal of $C_c(X)_F$ is fixed. We prove an ideal of $C_c(X)_F$ is an essential ideal if and only if it is free and also show that the set of all $(\mathcal{Z}_c)_F$-ideals and  the set of all $z^\circ$-ideals are identical (see Corollary \ref{SMP3,3.15}). Moreover we establish some results related to socle of $C_c(X)_F$. In the next section, we discuss structure spaces of $C_c(X)_F$ and prove that the set of all maximal ideals of $C_c(X)_F$ with hull-kernel topology is homeomorphic to the set of all $(\mathcal{Z}_c)_F$-ultrafilters on $X$ with Stone topology (see Theorem \ref{SMP3,4.3}) and also establish an analogue version of
	the Gelfand-Kolmogoroff theorem (see Theorem \ref{SMP3,4.4}). In Example \ref{SMP3,4.5}, we show that $\beta_\circ X$ and structure space of $C_c(X)_F$ are not  homeomorphic. In the next section, we establish that the  ring $C_c(X)_F$ properly contains the ring $C_c(X)$ and discuss some relations between $C_c(X)$ and $C_c(X)_F$.  In Section 6, we furnish some equivalent conditions when $C_c(X)_F$ is a Baer-ring. A space $X$ is called $F_cP$-space if $C_c(X)_F$ is regular and in Theorem \ref{SMP3,7.6}, some equivalent conditions of $F_cP$-space are proved. Finally, we introduce and study the main features of the zero divisor graph of $C_c(X)_F$ in Section 8.
	
	\section{Definitions and Preliminaries}
	
	For any $f\in C_c(X)_F$, $Z(f)=\{x\in X:f(x)=0\}$ is called zero set of the function $f$ and $Z[C_c(X)_F]$ aggregates of all zero sets in $X$. Then $Z[C_c(X)_F]=Z[C_c^*(X)_F]$, where $C_c^*(X)_F=\{f\in C_c(X)_F:f$ is bounded$\}$.
	
	Now we can easily check the following properties of zero sets:
	
	\begin{properties}
		Let  $f,g\in C_c(X)_F$ and for any $r\in\mathbb{R}$, $\underline{r}$ stands for constant function from $X$ to $\mathbb{R}$. Then
		\begin{itemize}
			\item[(i)] $Z(f)=Z(\vert f\vert)=Z(\vert f\vert\wedge\underline{1})=Z(f^n)$ (for all $n\in\mathbb{N})$.
			\item[(ii)] $Z(\underline{0})=X$ and $Z(\underline{1})=\emptyset$.
			\item[(iii)] $Z(f^2+g^2)=Z(f)\cap Z(g)=Z(\vert f\vert + \vert g\vert)$.
			\item[(iv)] $Z(f\cdot g)=Z(f)\cup Z(g)$.
			\item[(v)] $\{x\in X:f(x)\geq r\}$ and $\{x\in X:f(x)\leq r\}$ are zero sets in $X$.
		\end{itemize}
	\end{properties}
	
	Any two subsets $A$ and $B$ of a topological space $X$ are called completely separated [see 1.15,\cite{GJ}] if there exists a continuous function $f:X\rightarrow [0,1]$ such that $f(A)=\{0\}$ and $f(B)=\{1\}$. Analogously, we define $\mathcal{F}_c$-completely separated as follows:
	
	\begin{definition}
		Two subsets $A$ and $B$ of a topological space $X$ are said to be $\mathcal{F}_c$-completely separated in $X$ if there exists an element $f\in C_c(X)_F$ such that $f(A)=0$ and $f(B)=1$.
	\end{definition}
	
	\begin{theorem}\label{SMP3,2.3}
		Two subsets $A$ and $B$ of a topological space $X$ are  $\mathcal{F}_c$-completely separated if and only if they are contained in disjoint members of $Z[C_c(X)_F]$.
	\end{theorem}
	
	\begin{proof}
		Let $A,B$ be two $\mathcal{F}_c$-completely separated subsets of $X$. Then there exists $f\in C_c(X)_F$ with $f:X\rightarrow [0,1]$ such that $f(A)=\{0\}$ and $f(B)=\{1\}$. Take $Z_1=\{x\in X:f(x)\leq \frac{1}{5}\}$ and $Z_2=\{x\in X:f(x)\geq \frac{1}{3}\}$. The $Z_1,Z_2$ are disjoint zero sets in $Z[C_c(X)_F]$ and $A\subseteq {Z_1}, B\subseteq {Z_2}$. Conversely, let $A\subseteq Z(f),B\subseteq Z(g)$, where $Z(f)\cap Z(g)=\phi, f,g\in C_c(X)_F$. Let $h=\frac{f^2}{f^2+g^2}:X\rightarrow [0,1]$. Now $Z(f)\cap Z(g)=Z(f^2+g^2)=\phi$. Then we have $h\in C_c(X)_F$. Also $h(A)=\{0\},$ $h(B)=\{1\}$. This shows that $A, B$ are $\mathcal{F}_c$-separated in $X$.
	\end{proof}
	
	\begin{corollary}
		Any two disjoint zero sets in $Z[C_c(X)_F]$ are $\mathcal{F}_c$-completely separated in $X$.
	\end{corollary}
	
	\begin{theorem}
		If two disjoint subsets $A$ and $B$ of $X$ are $\mathcal{F}_c$-completely separated, then there exists a finite subset $F$ of $X$ such that $A\setminus F$ and $B\setminus F$ are completely separated in $X\setminus F$.
	\end{theorem}
	
	\begin{proof}
		Let $A$ and $B$ be $\mathcal{F}_c$-completely separated in $X$. Then by Theorem \ref{SMP3,2.3}, there exist two zero sets $Z(f_1)$ and $Z(f_2)$ such that $A\subseteq Z(f_1)$ and $B\subseteq Z(f_2)$. Since $f_1, f_2\in C_c(X)_F$, there is a finite subset $F$ such that $f_1,f_2\in C(X\setminus F)$. Now $A\setminus F\subseteq Z(f_1)\setminus F$ and $B\setminus F\subseteq Z(f_2)\setminus F$. Also, $Z(f_1)\setminus F$ and $Z(f_2)\setminus F$ are disjoint zero sets in $X\setminus F$. Thus by Theorem 1.15 in \cite{GJ}, $A\setminus F$ and $B\setminus F$ completely separated in $X\setminus F$.
	\end{proof}
	
	We recall that $C_c(X)$ be the ring of all real valued continuous functions with countable range and $C_c^*(X)=\{f\in C_c(X): f$ is bounded$\}$. Then we have the following lemma.
	
	\begin{lemma}
		For a topological space $X$, the following statements are hold.
		\begin{itemize}
			\item [(i)] $C_c(X)_F$ is a reduce ring.
			\item [(ii)] An element $f\in C_c(X)_F$ is unit if and only if $Z(f)=\emptyset$.
			\item [(iii)] Any element of $C_c(X)_F$ is zero divisor or unit.
			\item [(iv)] $C_c(X)_F=C_c^*(X)_F$ if and only if for any finite subset $F$ of $X$, $C_c(X\setminus F)=C_c^*(X\setminus F)$.
		\end{itemize}
	\end{lemma}
	
	\begin{proof}
		(i) It is trivial.
		
		(ii) Let $f\in C_c(X)_F$ be a unit. Then there exists $g\in C_c(X)_F$ such that $fg=\underline{1}$. Therefore $Z(f)=\emptyset$. Conversely, let $Z(f)=\emptyset$. Then $\frac{1}{f}\in C_c(X)_F$ is the inverse of $f$.
		
		(iii) Let $f\in C_c(X)_F$ and $Z(f)=\emptyset$. Then $f$ is a unit element. If $Z(f)\neq \emptyset$, then for $x\in Z(f)$, $\chi_{\{x\}}\in C_c(X)_F$ and $f\cdot\chi_{\{x\}}=0$ i.e., $f$ is a zero divisor.
		
		(iv) Suppose that $F$ is a finite subset of $X$ and $f\in C_c(X\setminus F)$. Now we define $g$ as
		\[  g(x)= \left\{
		\begin{array}{ll}
			0, & if~~ x\in F \\
			f(x), &  otherwise.\\
		\end{array}
		\right. \]
		Then $g\in C_c(X)_F=C_c^*(X)_F$ and $g\vert_{X\setminus F}=f$, hence $C_c(X\setminus F)=C_c^*(X\setminus F)$. Conversely, let $f\in C_c(X)_F$. Then there exists a finite subset $F$ of $X$ such that $f$ is continuous on $X\setminus F$. By hypothesis $f$ is bounded on $X\setminus F$. Therefore $f\in C_c^*(X)_F$.

	\end{proof}

	\section{$(\mathcal{Z}_c)_F$-filters and ideals of $C_c(X)_F$}
	
	Throughout the article, an ideal of $C_c(X)_F$ (or $C_c^*(X)_F$) always stands for a proper ideal.
	
	\begin{definition}
		A non-empty family $\mathcal{F}$ of subsets of $Z[C_c(X)_F]$ is called $(\mathcal{Z}_c)_F$-filter on $X$ if it satisfies the following three conditions:
		\begin{itemize}
			\item[(i)] $\phi\notin\mathcal{F}$.
			\item[(ii)] $Z_1,Z_2\in\mathcal{F}$ implies $Z_1\cap Z_2\in\mathcal{F}$.
			\item[(iii)] If $Z\in\mathcal{F}$ and $Z'\in Z[C_c(X)_F]$ such that $Z\subseteq Z'$, then $Z'\in\mathcal{F}$.
		\end{itemize}
	\end{definition}
	
	A $(\mathcal{Z}_c)_F$-filter on $X$ which is not properly contained in any $(\mathcal{Z}_c)_F$-filter on $X$ is called $(\mathcal{Z}_c)_F$-ultrafilter. A straight forward use of Zorn's lemma ensures that a $(\mathcal{Z}_c)_F$-filter on $X$ can be extended to a $(\mathcal{Z}_c)_F$-ultrafilter on $X$. There is an expected duality existing between ideals (maximal ideals) in $C_c(X)_F$ and the $(\mathcal{Z}_c)_F$-filters ($(\mathcal{Z}_c)_F$-ultrafilters) on $X$. This is realized by the following theorem.
	
	\begin{theorem}\label{SMP3,3.2}
		For the ring $C_c(X)_F$, the following statements are true.
		\begin{itemize}
			\item[(i)] If $I$ is an ideal of $C_c(X)_F$, then $Z[I]=\{Z(f):f\in I\}$ is a $(\mathcal{Z}_c)_F$-filter on $X$. Dually for any  $(\mathcal{Z}_c)_F$-filter $\mathcal{F}$ on $X$,  $Z^{-1}[\mathcal{F}]=\{f\in C_c(X)_F:Z(f)\in \mathcal{F}\}$ is an ideal (proper) in $C_c(X)_F$.
			\item[(ii)] If $M$ is a maximal ideal of $C_c(X)_F$ then $Z[M]$ is a $(\mathcal{Z}_c)_F$-ultrafilter on $X$. If $\mathcal{U}$ is a $(\mathcal{Z}_c)_F$-ultrafilter on $X$, then $Z^{-1}[\mathcal{U}]$ is a maximal ideal of $C_c(X)_F$. Furthermore the assignment: $M\mapsto Z[M]$ defines a bijection on the set of all maximal ideals in $C_c(X)_F$ and the aggregate of all $(\mathcal{Z}_c)_F$-ultrafilters on $X$.
		\end{itemize}
	\end{theorem}
	
	Like the notion of $z$-ideal in $C(X)$ (see 2.7 in \cite{GJ}), we now define $(\mathcal{Z}_c)_F$-ideal in $C_c(X)_F$.
	\begin{definition}
		An ideal $I$ of $C_c(X)_F$ is called $(\mathcal{Z}_c)_F$-ideal if $Z^{-1}Z[I]=I$.
	\end{definition}
	
	It follows from the Theorem \ref{SMP3,3.2}(ii) that each maximal ideal of $C_c(X)_F$ is a $(\mathcal{Z}_c)_F$-ideal; the converse of this statement is false as is shown by the following example.
	
	\begin{example}
		Let $I=\{f\in C_c(X)_F:f(0)=f(1)=0\}$. Then $I$ is a $(\mathcal{Z}_c)_F$-ideal in $C_c(X)_F$, which is not even a prime ideal in the  ring $C_c(X)_F$.
	\end{example}

	The next theorem describes maximal ideals of $C_c(X)_F$.
	
	\begin{theorem}
		For any $f\in C_c(X)_F$, we have $M_f=\{g\in C_c(X)_F:Z(f)\subseteq Z(g)\}$.
	\end{theorem}
	\begin{proof}
		Let $g\in M_f$ and $x\in Z(f)\setminus Z(g)$. Now $M_x=\{f\in C_c(X)_F:x\in Z(f)\}$ is a maximal ideal of $C_c(X)_F$ contains $f$ but does not contain $g$, a contradiction. Thus $Z(f)\subseteq Z(g)$. For reverse part, let $M$ be a maximal ideal of $C_c(X)_F$ which contains $f$ and $Z(f)\subseteq Z(g)$ for some $g\in C_c(X)_F$. Then we have $Z(g)\in Z[M]$ and this implies that $g\in Z^{-1}Z[M]$. Since $M$ is a $(\mathcal{Z}_c)_F$-ideal, $g\in M$.
	\end{proof}
	
	\begin{corollary}\label{SMP3,3.6}
		An ideal $I$ of $C_c(X)_F$ is a $(\mathcal{Z}_c)_F$-ideal if and  only if whenever $Z(f)\subseteq Z(g)$, where $f\in I$ and $g\in C_c(X)_F$, then $g\in I$.
	\end{corollary}
	
	The following two results are analogous to Theorem 2.9 and Theorem 2.11 respectively in \cite{GJ} and thus we state them without any proof.
	
	\begin{theorem}\label{SMP3,3.7}
		The following four statements are equivalent for a $(\mathcal{Z}_c)_F$-ideal $I$ in $C_c(X)_F:$
		
		\begin{itemize}
			\item[i)] $I$ is a prime ideal.
			
			\item[ii)] $I$ contains a prime ideal in $C_c(X)_F$.
			
			\item[iii)] For all $f,g\in C_c(X)_F, fg=0\Rightarrow f\in I$ or $g\in I$.
			
			\item[iv)] Given $f\in C_c(X)_F$, there exists $Z\in Z[I]$ on which $f$ does not change its sign.
		\end{itemize}
	\end{theorem}
	
	\begin{corollary}
		Each prime ideal in $C_c(X)_F$ is contained in a unique maximal ideal, in other words $C_c(X)_F$ is a Gelfand ring.
	\end{corollary}
	
	\begin{theorem}
		Sum of any two $(\mathcal{Z}_c)_F$-ideals in $C_c(X)_F$ is a $(\mathcal{Z}_c)_F$-ideal.
	\end{theorem}
	
	\begin{proof}
		Let $I$ and $J$ be two $(\mathcal{Z}_c)_F$-ideals of $C_c(X)_F$. Let $f\in I$, $g\in J$, $h\in C_c(X)_F$ and $Z(f+g)\subseteq Z(h)$. Then by Corollary \ref{SMP3,3.6}, it is enough to prove that $h\in I+J$. Now we can find a finite subset $F$ such that $f,g,h\in C_c(X)_F$. Define 
		\[  k(x)= \left\{
		\begin{array}{ll}
			0, & if~~ x\in (Z(f)\cap Z(g))\setminus F \\
			\frac{hf^2}{f^2+g^2}, &if ~~ x\in (X\setminus F)\setminus (Z(f)\cap Z(g)).\\
		\end{array}
		\right. \]
		\[  l(x)= \left\{
		\begin{array}{ll}
			0, & if~~ x\in (Z(f)\cap Z(g))\setminus F \\
			\frac{hg^2}{f^2+g^2}, &if ~~ x\in (X\setminus F)\setminus (Z(f)\cap Z(g)).\\
		\end{array}
		\right. \]
		
		Now we show that $k$ and $l$ are continuous on $X\setminus F$. Moreover, it is enough to show that $k$ and $l$ are continuous on $(Z(f)\cap Z(g))\setminus F$. For $x\in (Z(f)\cap Z(g))\setminus F$, $h(x)=0$ and for any $\epsilon>0$ there exists a neighbourhood $U$ of $x$ such that $h(U)\subseteq (-\epsilon,\epsilon)$. On the other hand $k(x)\leq h(x)$ and $l(x)\leq h(x)$ for all $x\in U$. Hence $k$ and $l$ are continuous on $X\setminus F$. Set $k^*(X\setminus  F)=k(X\setminus F)$, $k^*(F)=h(F)$ and $l^*(X\setminus  F)=l(X\setminus F)$, $l^*(F)=0$. Then $k^*,l^*\in C_c(X)_F$, $Z(f)\subseteq Z(k)\subseteq Z(k^*)$, $Z(g)\subseteq Z(l)\subseteq Z(l^*)$ and $h=k^*+l^*$. Since $I$ and $J$ are $(\mathcal{Z}_c)_F$-ideal of $C_c(X)_F$, $k^*\in I$ and $l^*\in J$. Therefore $h\in I+J$.
	\end{proof}
	
	\begin{corollary}
		Suppose that $\{I_k\}_{k\in S}$ is a collection of $(\mathcal{Z}_c)_F$-ideals of $C_c(X)_F$. Then $\sum\limits_{k\in S}I_k=C_c(X)_F$ or $\sum\limits_{k\in S}I_k$ is a $(\mathcal{Z}_c)_F$-ideal of $C_c(X)_F$.
	\end{corollary}
	
	In a reduced ring, every minimal prime ideal is also $z$-ideals (which is proved in
	\cite{Mason(1973)}). Now using this result, Theorem \ref{SMP3,3.7} and the above corollary, we have the following corollary.
	\begin{corollary}
		Let $\{P_i\}_{i\in I}$ be a collection of minimal prime ideals of $C_c(X)_F$. Then $\sum\limits_{i\in I}P_i=C_c(X)_F$ or $\sum\limits_{i\in I}P_i$ is a prime ideal of $C_c(X)_F$.
	\end{corollary}
	
	\begin{definition}
		An ideal $I$ of $C_c(X)_F$ is called fixed if $\cap Z[I]\neq\emptyset$. Otherwise it is called free.
	\end{definition}
	
	\begin{theorem}
		For any topological space $X$, the following statements are equivalent.
		\begin{itemize}
			\item[(i)] The space $X$ is finite.
			\item[(ii)] Every proper ideal of $C_c(X)_F$ (or $C_c^*(X)_F$) is fixed.
			\item[(iii)] Every maximal ideal of $C_c(X)_F$ (or $C_c^*(X)_F$) is fixed.
			\item[(iv)] Each $(\mathcal{Z}_c)_F$-filter on $X$ is fixed.
			\item[(v)] Each $(\mathcal{Z}_c)_F$-ultrafilter on $X$ is fixed.
		\end{itemize}
	\end{theorem}
	
	Let $M_A=\{f\in C_c(X)_F:A\subseteq Z(f)\}$, for a subset $A$ of $X$. Then $M_A$ is an ideal of $C_c(X)_F$ and $M_A=\bigcap\limits_{x\in A} M_x$, where $M_x=\{f\in C_c(X)_F: x\in Z(f)\}$ is a fixed maximal ideal of $C_c(X)_F$.
	
	\begin{theorem}\label{SMP3,3.14}
		The following statements are true.
		\begin{itemize}
			\item[(i)] For two ideals $I$ and $J$ of $C_c(X)_F$, $Ann(I)\subseteq Ann(J)$ if and only if $\bigcap Z[I]\subseteq \bigcap Z[J]$ if and only if $\bigcap COZ[J]\subseteq\bigcap COZ[I]$.
			\item[(ii)] For any subset $S$ of $C_c(X)_F$ we have $Ann(S)=M_{(\bigcup COZ[S])}=\{f\in C_c(X)_F: \bigcup COZ[S]\subseteq Z(f)\}$.
		\end{itemize}
	\end{theorem}
	
	\begin{proof}
		(i) Let $x\in \bigcap\limits_{f\in I}Z(f)$. Then $h=\chi_{\{x\}}\in C_c(X)_F$ and $x\in X\setminus Z(h)\subseteq\bigcap\limits_{f\in I}Z(f)$. Hence $fh=0$ for all $f\in I$. Then $h\in Ann(I)\subseteq Ann(J)$. Therefore $gh=0$ for each $g\in J$. Thus $x\in X\setminus Z(h)\subseteq \bigcap\limits_{g\in J}Z(g)$. Conversely, let $h\in Ann(I)$. Then $hf=0$ for all $f\in I$. This implies that $X\setminus Z(h)\subseteq \bigcap\limits_{f\in I}Z(f)$. Then by given hypothesis, $X\setminus Z(h)\subseteq Z(g)$ for each $g\in J$. Thus $gh=0$, for each $g\in J$, implies $Ann(I)\subseteq Ann(J)$.
		
		(ii) Let $f\in Ann(S)$. Then $fg=0$, for all $g\in S$. This shows that $\bigcup COZ[S]\subseteq Z(f)$ i.e., $f\in M_{(\bigcup COZ[S])}$. For the reverse part, let $f\in M_{(\bigcup COZ[S])}$. Then $X\setminus Z(g)\subseteq \bigcup COZ[S]\subseteq Z(f)$ for each $g\in S$. Thus $f\in Ann(S)$.
	\end{proof}
	
	A non-zero ideal in a commutative ring is said to be essential if it intersects every non-zero ideals non-trivially. Let $R$ be a commutative ring with unity. For $a\in R$, let $P_a$ be the intersection of all minimal prime ideals of $R$ containing $a$. Then an ideal $I$ of $R$ is called a $z^\circ$-ideal of $R$ if for each $a\in I$, $P_a\subseteq I$.
	
	We now state a well-known result that if $I$ is an ideal of a commutative reduced ring $R$, then $I$ is an essential ideal if and only if $Ann(I)=\{r\in R:rI=0\}=0$ (see \cite{MO(2008),OM(1985)} and Lemma 2.1 in \cite{Taherifar(2014)}).
	
	\begin{corollary}\label{SMP3,3.15}
		The following statement hold.
		\begin{itemize}
			\item[(i)] An ideal $I$ of $C_c(X)_F$ is an essential ideal if and only if $I$ is a free ideal.
			\item[(ii)] The set of all $(\mathcal{Z}_c)_F$-ideals and $z^\circ$-ideals of $C_c(X)_F$ are identical.
		\end{itemize}
	\end{corollary}
	
	\begin{proof}
		(i) It follows trivially from the above Theorem \ref{SMP3,3.14}.
		
		(ii) Clearly, every $z^\circ$-ideal is a $(\mathcal{Z}_c)_F$-ideal. Now let $I$ be a $(\mathcal{Z}_c)_F$-ideal and $Ann(f)\subseteq Ann(g)$. Then using Theorem \ref{SMP3,3.14}, we have $Z(f)\subseteq Z(g)$. Therefore $g\in I$. This completes the proof.
	\end{proof}
	
	It is well known that the intersection of all essential ideals or sum of all minimal prime ideals in a commutative ring with unity is called socle (see \cite{OM(1985)}).
	
	\begin{proposition}
		In a commutative ring with unity the following statements are true.
		\begin{itemize}
			\item[(i)] A non-zero ideal $I$ of $C_c(X)_F$ is minimal if and only if $I$ is generated by $\chi_{\{a\}}$, for some $a\in X$.
			\item[(ii)] A non-zero ideal $I$ of $C_c(X)_F$ is minimal if and only if $\vert Z[I]\vert=2$.
			\item[(iii)] The socle of $C_c(X)_F$ consists of all functions which vanish everywhere except on a finite subset of $X$.
			\item[(iv)] The socle of $C_c(X)_F$ is an essential ideal which is also free.
		\end{itemize}
	\end{proposition}
	
	\begin{proof}
		(i) Let $I$ be a non-zero ideal of $C_c(X)_F$ and $f$ be a non-zero element of $I$. Then there exists $a\in X$ such that $f(a)\neq 0$. Now $\chi_{\{a\}}=\frac{1}{f(a)}\chi_{\{a\}}f\in I$. This shows that $I$ is generated by $\chi_{\{a\}}$. Conversely, let $a\in X$. Then the ideal generated by $\chi_{\{a\}}$ is the set of all constant multiple of $\chi_{\{a\}}$, which is clearly a minimal ideal.
		
		(ii) Let us assume $\vert Z[I]\vert=2$ and $0\neq f\in I$ with $f(a)\neq 0$ for some $a\in X$. Then $\chi_{\{a\}}\in I$ and for any non-zero  element $g\in I$, $Z(g)=Z(\chi_{\{a\}})=X\setminus \{a\}$. Thus $g=g(a)\chi_{\{a\}}$ and hence $I$ is generated by $\chi_{\{a\}}$. Hence by (i), $I$ is minimal and the remaining part of the proof follows immediately.
		
		(iii) From $(i)$, we show that the socle of $C_c(X)_F$ is equal to the ideal generate by $\chi_{\{a\}}$'s which is equal to the set of all functions that vanishes everywhere except on a finite set.
		
		(iv) Clearly from $(i)$, any non-zero function $f$ has a non-zero multiple which is in the socle of $C_c(X)_F$. This implies that socle is essential. Then by Corollary \ref{SMP3,3.15}, the socle is a free ideal.
	\end{proof}
	
	\begin{corollary}
		For a topological space $X$, the the following statements are equivalent:
		\begin{itemize}
			\item[(i)] $X$ is finite set.
			\item[(ii)] $C_c(X)_F=Soc(C_c(X)_F)$, where $Soc(C_c(X)_F)$ is the socle of $C_c(X)_F$.
		\end{itemize}
	\end{corollary}
	
	\begin{proof}
		$(i)\implies (ii):$ Let $X=\{x_1,x_2,\cdots ,x_n\}$. Then $1=\sum\limits_{i=1}^n\chi_{\{x_i\}}\in Soc(C_c(X)_F)$, using result $(iii)$ of the above proposition. Therefore $C_c(X)_F=Soc(C_c(X)_F)$.
		
		$(ii)\implies (i):$ Let $C_c(X)_F=Soc(C_c(X)_F)$. Then  $\underline{1}\in Soc(C_c(X)_F)$. Hence from the above proposition, $X=X\backslash Z(\underline{1})$ is a finite set. This completes the proof.
	\end{proof}
	
	\begin{proposition}
		The following statements are true.
		\begin{itemize}
			\item[(i)] Any fixed maximal ideal of $C_c(X)_F$ is generated by an idempotent.
			\item[(ii)] Any non-maximal prime ideal of $C_c(X)_F$ is essential ideal.
		\end{itemize}
	\end{proposition}
	
	\begin{proof}
		(i) Let $\alpha\in X$ and consider the maximal ideal $M_\alpha$. Then for any $f\in M_\alpha$, $f=f(1-\chi_{\{\alpha\}})$. This shows that $M_\alpha$ generated by an idempotent $1-\chi_{\{\alpha\}}$.
		
		(ii) Let $P$ be a non-maximal prime ideal of $C_c(X)_F$. For each $\alpha\in X$, the ideal generated by $1-\chi_{\{\alpha\}}$ is maximal ideal (since ideal generated by $\chi_{\{\alpha\}}$ is minimal), then $1-\chi_{\{\alpha\}}\notin P$. Thus $\chi_{\{\alpha\}}\in P$ and hence $P$ contains in socle. Therefore $P$ is essential.
	\end{proof}
	
	The following theorem is an analogous version of Theorem 2.2 in \cite{MA(2014)}.
	
	\begin{theorem}
		Let $J_1=\{f\in C_c(X)_F:$ for all $g, Z(1-fg)$ is finite$\}$. Then $J_1$ is equal to the intersection of all essential maximal ideals (free maximal ideals) of $C_c(X)_F$. Also, for all $f\in C_c(X)_F$, $COZ(f)$ is a countable set.
	\end{theorem}
	\begin{proof}
		Let $f\in J_1$ and $M$ be an essential maximal ideal in $C_c(X)_F$ such that it does not contain $f$. Then for some $g\in C_c(X)_F$ and $m\in M$, we have $gf+m=1$. This implies that $m=1-gf$ and hence $Z(m)$ is finite. Take $h=m+\chi_{Z(m)}$. Then $h\in M$, since $m\in M$ and $\chi_{Z(m)}\in Soc(C_c(X)_F)\subseteq M$. On the other hand $h$ is invertible, a contradiction. Thus $f$ belongs to each essential maximal ideal of $C_c(X)_F$. Therefore $J_1\subseteq \bigcap\limits_{M\in S} M$, where $S$ is  the collection of all essential maximal ideals of $C_c(X)_F$. Next, let $f$ be any element in the intersection of all essential maximal ideals of $C_c(X)_F$. Let $g\in C_c(X)_F$ be such that $Z(1-gf)$ infinite. This implies that for any $s\in Soc(C_c(X)_F)$ and any $t\in C_c(X)_F$, the function $s+t(1-gf)$ has a zero and thus it can not be equal to $1$. Then the ideal $Soc(C_c(X)_F)+<1-gf>$ is a proper essential ideal. Thus there exists an essential maximal ideal $M$ containing it. Therefore $1-gf\in M$ and $f\in M$, a contradiction. Hence $Z(1-gf)$ is finite. This completes the proof.
		
		For the second part, we define $F_n=\{x\in X:\vert f(x)\vert\geq \frac{1}{n}\}$, for each $n\in \mathbb{N}$. Since $COZ(f)=\bigcup\limits_{n=1}^\infty F_n$, it is enough to show that $F_n$ is a finite set for any $n\in \mathbb{N}$. If possible, let $F_n$ be infinite for some $n\in \mathbb{N}$. Let $g:\mathbb{R}\rightarrow \mathbb{R}$  be a continuous function such that $g(x)=\frac{1}{x}$ if $\vert x\vert\geq\frac{1}{n}$ and take $h=g\circ f$. Then $h\in C_c(X)_F$ and $F_n\subseteq Z(1-hf)$. This implies that $Z(1-fh)$ is infinite, a contradiction.
	\end{proof}
	
	\section{Structure space of $C_c(X)_F$}
	
	Let $Max(C_c(X)_F)$ be the structure space of $C_c(X)_F$ i.e., $Max(C_c(X)_F)$ is the set of all maximal ideals of $C_c(X)_F$ equipped with hull-kernel topology. Then $\{\mathcal{M}_f:f\in C_c(X)_F\}$ form a base for closed sets with this hull-kernel topology (see 7M \cite{GJ}), where $\mathcal{M}_f=\{M\in C_c(X)_F:f\in M\}$. Using Theorem 1.2 of \cite{GA(1971)}, we have $Max(C_c(X)_F)$ is a Hausdorff compact space. It is checked that the structure space of $C_c(X)_F$ is identical with the set of all $(\mathcal{Z}_c)_F$-ultrafilters on $X$ with Stone topology.
	
	Let $\beta_\circ^F X$ be an index set for the family of all $(\mathcal{Z}_c)_F$-ultrafilters on X i.e., for each $p\in\beta_\circ^F X$, there exists a $(\mathcal{Z}_c)_F$-ultrafilter on $X$, which is denoted by $\mathcal{U}^p$. For any $p\in X$, we can find a fixed $(\mathcal{Z}_c)_F$-ultrafilter $\mathcal{U}_p$ and set $\mathcal{U}_p=\mathcal{U}^p$. Then we can think $X$ is a subset of $\beta_\circ^F X$.
	
	Now we wish to define a topology on $\beta_\circ^F X$. Let $\beta=\{\overline{Z}:Z\in Z[C_c(X)_F]\}$, where $\overline{Z}=\{p\in\beta_\circ^F X: Z\in \mathcal{U}^p\}$. Then $\beta$ is a base for closed sets for some topology on $\beta_\circ^F X$. Since X belongs to every $(\mathcal{Z}_c)_F$-ultrafilters on $X$, then we have $\overline{X}=\beta_\circ^F X$. We can easily check that $\overline{Z}\cap X=Z$ and for $Z_1,Z_2\in Z[C_c(X)_F]$ with $Z_1\subseteq Z_2$, then $\overline{Z_1}\subseteq\overline{Z_2}$. This leads to the following result.
	
	\begin{theorem}
		For $Z\in Z[C_c(X)_F]$, $\overline{Z}=Cl_{\beta_\circ^F X}{Z}$.
	\end{theorem}
	
	\begin{proof}
		Let $Z\in Z[C_c(X)_F]$ and $\overline{{{Z_1}}}\in\beta$ be such that $Z\subseteq\overline{Z_1}$. Then $Z\subseteq\overline{Z_1}\cap X=Z_1$. This implies $\overline{Z}\subseteq\overline{Z_1}$. Therefore $\overline{Z}$ is the smallest basic closed set containing Z. Hence $\overline{Z}=Cl_{\beta_\circ^F X}Z$.
	\end{proof}
	
	\begin{corollary}
		$X$ is a dense subset of $\beta_\circ^F X$.
	\end{corollary}
	\begin{proof}
		Since $X$ is a member of every $(\mathcal{Z}_c)_F$-ultrafilter on $X$, it implies that $\overline{X}=\beta_\circ^F X$.
	\end{proof}
	
	Now, we want to show that $Max(C_c(X)_F)$ and $\beta_\circ^F X$ are homeomorphic.
	
	\begin{theorem}\label{SMP3,4.3}
		The map $\phi: Max(C_c(X)_F)\rightarrow \beta_\circ^F X$, defined by  $\phi(M)= p$ is a homeomorphism, where $Z[M]=\mathcal{U}^p$.
	\end{theorem}
	
	\begin{proof} 
		The map $\phi$ is  bijective  by Theorem \ref{SMP3,3.2} (ii). Basic closed set of $Max(C_c(X)_F)$ is of the form $\mathcal{M}_f=\{M\in Max(C_c(X)_F):f\in M\}$, for some $f\in C_c(X)_F$. Now $M\in \mathcal{M}_f \Leftrightarrow f\in M\Leftrightarrow Z(f)\in Z[M]$ (Since maximal ideal is a $(\mathcal{Z}_c)_F$-ideal$)\Leftrightarrow Z(f)\in\mathcal{U}^p \Leftrightarrow p\in\overline{Z(f)}$. Thus $\phi(\mathcal{M}_f)=\overline{Z(f)}$. Therefore $\phi$ interchanges basic closed sets of $Max(C_c(X)_F)$ and $\beta_\circ^F X$. Hence $Max(C_c(X)_F)$ is homeomorphic to $\beta_\circ^F X$.
	\end{proof}
	
	Now we prove the following theorem which is an analogous version of the Gelfand-Kolmogoroff Theorem 7.3 \cite{GJ}.  
	
	\begin{theorem}\label{SMP3,4.4}
		Every maximal ideal of $C_c(X)_F$ is of the form $M^p=\{f\in C_c(X)_F:p\in Cl_{\beta_\circ^F X}Z(f)\}$, for some $p\in\beta_\circ^F X$.
	\end{theorem}
	
	\begin{proof}
		Let $M$ be any maximal ideal of $C_c(X)_F$. Then $Z[M]$ is a $(\mathcal{Z}_c)_F$-ultrafilter on $X$. Thus $Z[M]=\mathcal{U}^p$, for some $p\in \beta_\circ^F X$. So, $f\in M \Leftrightarrow Z(f)\in Z[M]$ as $M$ is a $(\mathcal{Z}_c)_F$-ideal $ \Leftrightarrow Z(f)\in Z[M]=\mathcal{U}^p \Leftrightarrow p\in \overline{Z(f)}=Cl_{\beta_\circ^F X}Z(f)$. Hence $M=\{f\in C_c(X)_F:p\in Cl_{\beta_\circ^F X} Z(f)\}$ and so we can write $\{f\in C_c(X)_F:p\in Cl_{\beta_\circ^F X} Z(f)\}=M^p$, $p\in \beta_\circ^F X$. This completes the proof.
	\end{proof}
	
	We know that the structure space of $C_c(X)$ is homeomorphic to $\beta_\circ X$ ($\equiv$ the Banaschewski compactification of a zero dimensional space $X$). Also, it is interesting to note that the structure space of $C_c(X)$ and the structure space of $C_c(X)_F$  are same if $X$ is equipped with discrete topology. The following example  shows that  these spaces may not be homeomorphic  to each other. 
	
	\begin{example}\label{SMP3,4.5}
		Take $X=\{\frac{1}{n}:n\in\mathbb{N}\}\cup\{0\}$. Consider $(X,\tau_u)$, where $\tau_u$ is the subspace topology on $X$ of the real line. Since $X$ is a zero dimensional space (it has a base of clopen sets), then the Stone-$\check{C}$ech compactification of $X$, $\beta X=\beta_\circ X$ (see \cite{JR(1988)}, subsection 4.7). Again $X$ is a compact space implies $X$ is homeomorphic to $\beta X=\beta_\circ X$. On the other hand since $X$ contains only one non-isolated point and it is a countable set, $C_c(X)_F=\mathbb{R}^X= C(X,\tau_d)$, where $C(X,\tau_d)$ is the rings of continuous functions with discrete topology $\tau_d$. Hence $\beta_\circ^F X$ is the Stone-$\check{C}$ech compactification of $X$ equipped with the discrete topology and the cardinality of $\beta_\circ^F X$  is equal to $\vert \beta \mathbb{N}\vert =2^c$ (see 9.3 in \cite{GJ}), where $\beta\mathbb{N}$ is the  Stone-$\check{C}$ech compactification of the set $\mathbb{N}$ of all natural numbers. Now the cardinality of $\beta X$ is $\aleph_\circ$ implies that $\beta_\circ^F X$ is not homeomorphic to $\beta X=\beta_\circ X$.
	\end{example}

	\section{$C_c(X)_F$ and $C_c(X)$}
	In this section,  we shall discuss relation between $C_c(X)_F$ and $C_c(X)$. It is interesting to see that the  ring $C_c(X)_F$ properly contains the ring $C_c(X)$. In fact, for any topological space $X$, let $x$ be a non-isolated point of $X$. Then $\chi_{\{x\}}\in C_c(X)_F$, but $\chi_{\{x\}}\not \in C_c(X)$. \\
	 
	Now we recall that an over ring $S$ of a reduced ring $R$ is called a quotient ring of $R$ if for any non-zero element $s\in S$, there is an element $r\in R$ such that $0\neq sr\in R$.
	\begin{theorem}
		For any topological space $X$, the following statements are equivalent.
		\begin{itemize}
			\item[(i)] $C_c(X)_F=C_c(X)$.
			\item[(ii)] $X$ is a discrete space.
			\item[(iii)] $C_c(X)_F$ is a quotient ring of $C_c(X)$.
		\end{itemize}
	\end{theorem}
	
	\begin{proof}
		$(i)\Leftrightarrow (ii)$ Let $X$ be a discrete space, then clearly $C_c(X)_F=C_c(X)$. Next we assume that $C_c(X)_F=C_c(X)$. Then for each $x\in X$, $\chi_{\{x\}}\in C_c(X)_F=C_c(X)$, a continuous map. Then $\{x\}$ is an isolated point. Therefore $X$ is a discrete space.
		
		$(ii)\Rightarrow (iii)$ It is trivial.
		
		$(iii)\Rightarrow (ii)$ Let $x_\circ\in X$. Then $\chi_{\{x_\circ\}}\in C_c(X)_F$. Then by given hypothesis there exists a function $f\in C_c(X)$ such that $0\neq f\cdot\chi_{\{x_\circ\}}\in C_c(X)$. Now $f(x)\chi_{\{x_\circ\}}(x)=f(x_\circ)\chi_{\{x_\circ\}}(x)$ for all $x\in X$. Hence $\chi_{\{x_\circ\}}$ is a continuous function. This implies that $\{x_\circ\}$ is an isolated point. Therefore $X$ is a discrete space.
	\end{proof}
	
	\begin{lemma}\label{SMP3,5.2}
		Let $\phi:C_c(X)_F\rightarrow C_c(Y)$ be a ring isomorphism, where $X$ and $Y$ are  two topological spaces. Then the following statements are true.
		\begin{itemize}
			\item [(i)] Both $\phi$ and $\phi^{-1}$ are order preserving function.
			\item[(ii)] Both $\phi$ and $\phi^{-1}$ preserve constant functions (and their values).
			\item [(iii)] For any $x_\circ\in X$ there is an element $y_\circ\in Y$ such that $\phi(\chi_{\{x_\circ\}})=\chi_{\{y_\circ\}}$ and $\phi(f)(y_\circ)=f(x_\circ)$ for any $f\in C_c(X)_F$.
		\end{itemize}
	\end{lemma}
	
	\begin{proof}
		(i) Let $f\in C_c(X)_F$ and $f\geq 0$, then there exists an element $g\in C_c(X)_F$ such that $f=g^2$. Hence $\phi(f)=(\phi(g))^2\geq 0$. Thus $\phi$ is order preserving. Similarly, $\phi^{-1}$ is also order preserving.
		
		(ii) We can easily check that $\phi(1)=1$ and use this we can calculate that $\phi$ maps any constant function with rational value to same constant and together with the order preserving property of $\phi$ shows that $\phi$ preserves constant functions. Similarly, constant functions are preserved by $\phi^{-1}$.
		
		(ii) For $x\in X$, $\phi(\chi_{\{x_\circ\}})$ is an idempotent element of $C_c(Y)$. Then there is a clopen subset $A$ of $Y$ such that $\phi(\chi_{\{x_\circ\}})=\chi_A$. Clearly, $A$ is non-empty. Now, we prove that $A$ is a singleton set. If possible, let $y,z\in A$. Then there exists a continuous function $f:Y\rightarrow [0,1]$ such that $f(y)=0$ and $f(z)=1$. Take  $g=min\{f,\chi_A\}\in C_c(Y)$. Then $0\leq g\leq \chi_A$ and we have $0\leq \phi^{-1}(g)\leq\chi_{\{x\}}$, consequently $\phi^{-1}(g)=k\chi_{\{x\}}$ for some real number $k$. Hence $g=k \chi_A$. Now $g(y)=k \chi_A(y)$ implies that $k=0$ and $g(z)=\chi_A(z)$ implies that $k=1$, a contradiction. Therefore $A$ is a singleton set.
		
		To prove the second part, let $x_\circ\in X$, $f\in C_c(X)_F$ and $\phi(\chi_{\{x_\circ\}})=\chi_{\{y_\circ\}}$ for some $y_\circ \in Y$. If possible, let $\phi(f-f(x_\circ))(y_\circ)\neq 0$. Then we have $\phi(f-f(x_\circ))^2\geq \frac{d^2}{3}\chi_{\{y_\circ\}}=\phi(\frac{d^2}{3}\chi_{\{x_\circ\}})$, where $d=\phi(f-f(x_\circ))(y_\circ)$. Thus $(f-f(x_\circ))^2\geq \frac{d^2}{3}\chi_{\{x_\circ\}}$, contradiction when evaluated at $x_\circ$. Therefore $\phi(f)(y_\circ)=f(x_\circ)$.
	\end{proof}
	
	\begin{theorem}
		For a topological spaces $X$. There exists a topological space $Y$ such that $C_c(X)_F\cong C_c(Y)$ if and only if set of all non-isolated points of $X$ is finite.
	\end{theorem}
	
	\begin{proof}
		Let $\{x_1,x_2,\cdots, x_n\}$ be the set of all non-isolated points of $X$. Set $Y=X\setminus\{x_1,x_2,\cdots,x_n\}$. Then for each $f\in C_c(X)_F$, we have $f\vert_Y\in C_c(Y)$ and $f\mapsto f\vert_Y$ is the required isomorphism. For the reverse part, let $\phi:C_c(X)_F\rightarrow C_c(Y)$ be a ring isomorphism. If possible, let $\{x_1,x_2,\cdots\}$ be infinite number of isolated points and without loss of generality for each $i$, $x_i$ be a limit point of $X\setminus\{x_1,x_2,\cdots\}$. Then by Lemma \ref{SMP3,5.2}(iii), we get $\{y_1,y_2,\cdots\}$ such that $\phi(\chi_{\{x_i\}})=\chi_{\{y_i\}}$ for each $i$. Let $g=\sum\limits_{i=1}^\infty\chi_{\{y_i\}}$ (which is an element of $C_c(Y)$) and $f$ be an inverse image of $g$ under $\phi$. Then $f(x)=0$ if and only if $x\notin\{x_1,x_2,\cdots\}$ and hence discontinuity set of $f$ is infinite, which contradicts that $f\in C_c(X)_F$.
	\end{proof}
	
	\section{$C_c(X)_F$ as a Baer-ring}
	In this section, we give a characterization of $C_c(X)_F$ as a Baer-ring. A ring $R$ is called a Baer-ring if annihilator of every non-empty ideal is generated by an idempotent. A ring $R$ is said to be a $SA$-ring if for any two ideals $I$ and $J$, there exists an ideal $K$ of $R$ such that $Ann(I)\cap Ann(J)=Ann(K)$. A ring $R$ is called an $IN$-ring if for any two ideals $I$ and $J$, $Ann(I\cap J)=Ann(I)\cap Ann(J)$. Clearly, any $SA$-ring is always an $IN$-ring.
	
	Next lemma states a characterization of $IN$-ring when it is also a reduced ring.
	
	\begin{lemma}\label{SMP3,6.1}
		(\cite{Taherifar(2014)(1)}) Let $R$ be a reduced ring. Then the following statements are equivalent.
		\begin{itemize}
			\item[(i)] For any two orthogonal ideals $I$ and $J$ of $R$, $Ann(I)+Ann(J)=R$.
			\item[(ii)] For any two ideals $I$ and $J$ of $R$, $Ann(I)+Ann(J)=Ann(I\cap J)$.
		\end{itemize}
	\end{lemma}
	
	Furthermore, the next lemma gives some equivalent conditions of Baer-ring.
	
	\begin{lemma}\label{SMP3,6.2}
		(\cite{GMA(2015)}) Let $R$ be a commutative reduced ring. Then the following statements are equivalent.
		\begin{itemize}
			\item[(i)] $R$ is a Baer-ring.
			\item[(ii)] $R$ is a $SA$-ring.
			\item[(iii)] The space of prime ideals of $R$ is extremally disconnected.
			\item[(iv)] $R$ is a $IN$-ring.
		\end{itemize}
	\end{lemma}
	
	Now we wish to establish some equivalent conditions when $C_c(X)_F$ is an $IN$-ring, $SA$-ring or a Baer-ring. For this purpose we first prove the following  lemma.
	
	\begin{lemma}\label{SMP3,6.3}
		For any subset $A$ of a space $X$, there exists a subset $S$ of $C_c(X)_F$ such that $A=\bigcup COZ[S]=\bigcup\{COZ(f):f\in S\}$.
	\end{lemma}
	\begin{proof}
		This follows immediately, since $A=\bigcup \{\chi_{\{x\}}:x\in A\}$ and $\chi_{\{x\}}\in C_c(X)_F$ for all $x\in X$.
	\end{proof}
	
	Now, in this situation, we are ready to prove the following equivalent conditions.
	
	\begin{theorem}
		The following statements are equivalent.
		\begin{itemize}
			\item[(i)] Any two disjoint subsets of $X$ are $\mathcal{F}_c$-completely separated.
			\item[(ii)] $C_c(X)_F$ is an $IN$-ring.
			\item[(ii)] $C_c(X)_F$ is a $SA$-ring.
			\item[(iv)] $C_c(X)_F$ is a Baer-ring.
			\item[(v)] The space of prime ideals of $C_c(X)_F$ is an extremally disconnected space.
			\item[(vi)] Any subset of $X$ is of the form $COZ(e)$ for some idempotent $e\in C_c(X)_F$.
			\item[(vii)] For any subset $A$ of $X$, there exists a finite subset $F$ of $X$ such that $A\setminus F$ is clopen in $X\setminus F$.
		\end{itemize}
	\end{theorem}
	\begin{proof}
		From Lemma \ref{SMP3,6.2}, statements $(ii), (iii), (iv)$ and $(v)$ are equivalent.
		
		$(i)\implies (ii)$: Let $I$ and $J$ be two orthogonal ideals of $C_c(X)_F$. Then $IJ=0$ and $\bigcup COZ[I]$, $\bigcup COZ[J]$ are two disjoint subsets of $X$. Now, by given hypothesis there exist two elements $f_1,f_2\in C_c(X)_F$ such that $\bigcup COZ[I]\subseteq Z(f_1)$ and $\bigcup COZ[J]\subseteq Z(f_2)$ with $Z(f_1)\cap Z(f_2)=\emptyset$. This implies that $f_1\in Ann(I)$, $f_2\in Ann(J)$ and $Z(f_1^2+f_2^2)=\emptyset$. Then $f_1^2+f_2^2$ is a unit element in $Ann(I)+Ann(J)$. Hence $Ann(I)+Ann(J)=C_c(X)_F$. Thus by Lemma \ref{SMP3,6.1}, $C_c(X)_F$ is an $IN$-ring.
		
		$(ii)\implies (i)$: Suppose that $A$ and $B$ are two disjoint subsets of $X$. Then by Lemma \ref{SMP3,6.3}, there are two subset $S_1$ and $S_2$ of $C_c(X)_F$ such that $A=\bigcup COZ[S_1]$ and $B=\bigcup COZ[S_2]$. Let $I$ and $J$ be two ideals generated by $S_1$ and $S_2$ respectively. Then we have $\bigcup COZ[I]\cap \bigcup COZ[J]=A\cap B=\emptyset$. This implies that $IJ=0$ i.e., $I$ and $J$ are orthogonal ideals of $C_c(X)_F$. Then by Lemma \ref{SMP3,6.1}, $Ann(I)+Ann(J)=C_c(X)_F$. So there exist $h_1\in Ann(I)$ and $h_2\in Ann(J)$ such that $1=h_1+h_2$, a unit element. Then $Z(h_1)$ and $Z(h_2)$ are disjoint. Since $h_1\in Ann(I)$, $A=\bigcup COZ[S_1]=\bigcup COZ[I]\subseteq Z(h_1)$. Similarly, $B\subseteq Z(h_2)$. This proves $(i)$.
		
		$(iv)\implies (vi)$: Let $A$ be a subset of $X$. Then by Lemma \ref{SMP3,6.3}, there exists a subset $S$ of $C_c(X)_F$ such that $A=\bigcup COZ[S]$. Let $I$ be an ideal generated by $S$. Then by given hypothesis, there exists an idempotent $e\in C_c(X)_F$ such that $Ann(e)=Ann(I)$. Thus by using Theorem \ref{SMP3,3.14}, we have $A=\bigcup COZ[I]=COZ(e)$.
		
		$(vi)\implies (iv)$:  Suppose that $I$ is an ideal of $C_c(X)_F$. Then by given hypothesis there is an idempotent $e\in C_c(X)_F$ such that $\bigcup COZ[I]=COZ(e)$. Then by Theorem \ref{SMP3,3.14}, $Ann(I)=Ann(e)=(1-e)C_c(X)_F$. Hence $C_c(X)_F$ is a Baer-ring.
		
		$(vi)\implies (vii)$: Let $A=COZ(e)$, for some idempotent $e\in C_c(X)_F$. Then $e\in C(X\setminus F)$ for some finite subset $F$ of $X$. Now $A\setminus F=COZ(e)\cap (X\setminus F)$ is clopen in $X\setminus F$.
		
		$(vii)\implies (vi)$: Trivial.
	\end{proof}

	\section{$F_cP$-apace}
	
	\begin{definition}
		A commutative ring $R$ with unity is said to be a von Neumann regular ring or simply a regular ring if for each $a\in R$, there exists $r\in R$ such that $a=a^2 r$.
	\end{definition}	
	 We recall that X is a P-space if and only if C(X) is regular ring [4J,\cite{GJ}].\\
	 
	 Now we define $F_cP$-Space as follows.
	\begin{definition}
		A space X is called $F_cP$-space if $C_c(X)_F$ is regular.
	\end{definition}
	
	\begin{example}
		Consider the space $X=\{0,1,\frac{1}{2},\frac{1}{3}, \cdots \}$ (endowed with the subspace topology from the usual topology on the real line $\mathbb{R}$). It is clear that $C_c(X)_F=\mathbb{R}^X$, which means that X is an $F_cP$-space. On the other hand by 4K.1 \cite{GJ}, X is not a $P$-space.
	\end{example}
	
	Next example shows that for every $P$-space $X$, $C_c(X)_F$ may not be regular i.e., not a $F_cP$-apace.
	
	\begin{example} Let the set of rational numbers $\mathbb{Q}$ be the subspace of the real line $\mathbb{R}$. Let $X=\mathbb{Q}^*=\mathbb{Q}\cup\{\infty\}$, one point compactification of $\mathbb{Q}$. Then every continuous function $f:\mathbb{Q}^*\rightarrow \mathbb{R}$ is a constant function. Hence $C(\mathbb{Q}^*)$ is isomorphic to $\mathbb{R}$, a regular ring. Hence $X$ is a $P$-space.
		
		But we wish to show that $C_c(X)_F$ is not a regular ring.
		Let $f:\mathbb{Q}^*\rightarrow\mathbb{R}$ be defined as,
		\[  f(x)= \left\{
		\begin{array}{ll}
			cos(\frac{\pi x}{2}), & if~ x\in \mathbb{Q} \\
			2, & if~ x=\infty.\\
		\end{array}
		\right. \]
		Then $f\in C_c(X)_F$. If possible, let there exists a $g\in C_c(X)_F$ such that $f=f^2 g$. Then $g(x)=\frac{1}{f(x)}$ when $f(x)\neq 0$ and discontinuity set of $g$, $D_g\supseteq\{1,-1,3,-3,5,-5,\cdots\}$. This shows that $g\notin C_c(X)_F$. Hence $C_c(X)_F$ is not regular i.e., $X$ is not a $F_cP$-apace.
	\end{example}
	
	However, if we consider $X$ as a Tychonoff space, then the following statement is true.
	
	\begin{theorem}
		If $X$ is a P-space, then it is also a $F_cP$-space.
	\end{theorem}
	
	\begin{proof}
		Let $f\in C_c(X)_F$. Then $f\in C(X\setminus F)$ for a finite subset $F$ of $X$. Since subspace of a $P$-space is $P$-space (see 4K.4 in \cite{GJ}), $X\setminus F$ is a $P$-space. So $C(X\setminus F)$ is regular. This implies that $(f\vert_{X\setminus F})^2.g=f\vert_{X\setminus F}$, for some $g\in C(X\setminus F)$. Now, we define $g^*:X\rightarrow \mathbb{R}$ by
		\[  g^*(x)= \left\{
		\begin{array}{ll}
			g(x), & if~ x\in (X\setminus Z(f))\setminus F \\
			0, & if~ x\in (Z(g)\cap Z(f))\setminus F\\
			0, & if~ x\in (Z(f)\setminus Z(g))\setminus F\\
			\frac{1}{f(x)}, & if ~ x\in F\setminus Z(f)\\
			0, & if ~ x\in Z(f)\cap F.\\
		\end{array}
		\right. \]
Since $f\vert_{X\setminus F}$ is continuous, then $(X\setminus Z(f))\setminus F$ is open in $X\setminus F$. Also, $X\setminus F$ is open in $X$ and hence $(X\setminus Z(f))\setminus F$ is open in $X$. Now, $(Z(g)\cap Z(f))\setminus F$ is $G_\delta$-set in $X\setminus F$ and $X\setminus F$ is  $G_\delta$-set in $X$, then $(Z(g)\cap Z(f))\setminus F$ is $G_\delta$-set in $X$. Hence it is open in $X$ (see 4J.4 \cite{GJ}). Again, $(Z(f)\setminus Z(g))\setminus F=(Z(f)\cap (Z(g))^c)\setminus F$ is $G_\delta$-set in $X\setminus F$ and hence open in $X$ (using 4J.4 \cite{GJ}). Then using Pasting lemma, we can easily observe that $g^*$ is well defined and continuous on $X\setminus F$. Thus $g^\star\in C_c(X)_F$ and $f=f^2.g^\star$. So it is a $F_cP$-space.
	\end{proof}

	\

	\begin{theorem}\label{SMP3,7.6}
		The following statements are equivalent.
		\begin{itemize}
			\item[(i)] The space $X$ is an $F_cP$-space.
			\item[(ii)] For any $Z\in Z[C_c(X)_F]$, there exists a finite subset $F$ in $X$ such that $Z\setminus F$ is a clopen subset in $X\setminus F$.
			\item[(iii)] $C_c(X)_F$ is a $PP$-ring, that is annihilator of every element is generated by an idempotent.
		\end{itemize}
	\end{theorem}
	
	\begin{proof}
		$(i)\implies (ii)$: Let $Z(f)\in Z[C_c(X)_F]$. By $(i)$, there exists $g\in C_c(X)_F$ such that $f^2g=f$. Also, there is a finite subset $F$ of $X$ such that $f,g\in C(X\setminus F)$. Hence for any $x\in X\setminus F$ we have $f(x)^2g(x)=f(x)$. Therefore $Z(f\vert_{X\setminus F})\cup Z((1-fg)\vert_{X\setminus F})=X\setminus F$. On the other hand $Z(f\mid_{X\setminus F})\cap Z((1-fg)\vert_{X\setminus F})=\phi$. These shows that $Z(f\vert_{X\setminus F})=Z(f)\setminus F$ is clopen in $X\setminus F$.
		
		$(ii)\implies (iii)$: Let  $f\in C_c(X)_F$. Then there is a finite subset $F$ of $X$ such that $Z(f)\setminus F$ is clopen in $X\setminus F$. So $Z(f)\setminus F =Z(e)$ for some idempotent $e\in C_c(X\setminus F)$. Therefore $Z(f)=Z(e^\star)$ where $e^\star\vert_{X\setminus F}=e, e^\star$ is zero on $Z(f)\cap F$ and $e^\star$ is equal to $1$ on $F\setminus Z(f)$. Then by Theorem \ref{SMP3,3.14}, we have $Ann(f)=Ann(e^\star)=(1-e^\star)C_c(X)_F$, i.e., $C_c(X)_F$ is a $PP$-ring.
		
		$(iii)\implies(i)$: Assume that $f\in C_c(X)_F$. By hypothesis, there is an idempotent $e\in C_c(X)_F$ such that $Ann(e)=Ann(f)$. By Theorem \ref{SMP3,3.14}, $Z(e)=Z(f)$. Now, $F$ is a finite subset such that $f,e\in C_c(X\setminus F)$. Then $Z(f)\setminus F=Z(e)\setminus F$ is a clopen subset in $X\setminus F$. Now, we define $f^\star:X\rightarrow\mathbb{R}$ by  
		\[  f^*(x)= \left\{
		\begin{array}{ll}
			0, & if~ x\in Z(f)\\
			\frac{1}{f(x)}, & otherwise.\\
		\end{array}
		\right. \]
		Then $f^\star\in C_c(X)_F$ and $f^2f^\star=f$. Thus $C_c(X)_F$ is a regular ring i.e., X is an $F_cP$-space.
		
	\end{proof}

	\section{Zero divisor graph of $C_c(X)_F$}
	
	Consider the graph $\Gamma(C_c(X)_F)$ of the ring $C_c(X)_F$ with the vertex set $V,$ the collection of  all nonzero zero divisors in the ring $C_c(X)_F$ and two vertices $f,g $ are adjacent if and only if $fg=0$ on $X$. We recall that for $f\in C_c(X)_F$, $Z(f)=\{x\in X: f(x)=0\}$ is called the zero set of $f$. For a $T_1$ space  $X$ and $x\in X$, the characteristic function $\chi_{\{x\}}$ defined by $\chi_{\{x\}}(y)=0$ if $y\neq x$ and $\chi_{\{x\}}(x)=1$ is a member of $C_c(X)_F$.
	\par The following result provides a condition of the adjacency of two vertices of $\Gamma(C_c(X)_F)$ in terms of their zero sets.

	\begin{lemma}\label{deg1}
		Two vertices $f,g$ in the graph $\Gamma(C_c(X)_F)$ are adjacent if and only if $Z(f)\cup Z(g)=X$.
	\end{lemma}
	\begin{proof}
		To begin with, let us assume that $Z(f)\cup Z(g)=X$. Then  $fg=0$. So $f$ and $g$ are adjacent in $\Gamma(C_c(X)_F).$ 
		\par Conversely, let $f$ and $g$ be adjacent in $\Gamma(C_c(X)_F)$. Then  $fg=0$. This implies $Z(0)=Z(fg)=Z(f)\cup Z(g)=X$.
	\end{proof}
	\begin{lemma} \label{deg 2}
		For any two vertices $f,g$ there is another vertex $h$ in the graph $\Gamma(C_c(X)_F)$ which is adjacent to both $f$ and $g$ if and only if $Z(f)\cap Z(g)\neq \emptyset$. 
	\end{lemma}
	
	\begin{proof}
		Firstly, we consider that there is a vertex $h$ such that $h$ is adjacent to both $f$ and $g$. Then $hf=0$ and $hg=0$.	As $h$ is non-zero, there exists a point $x_\circ\in X$ such that $h(x_\circ)\neq 0$. Then obviously, $f(x_\circ)=0$ and $g(x_\circ)=0$. Hence $x_\circ\in Z(f)\cap Z(g)$. Thus $Z(f)\cap Z(g)\neq \emptyset$. 
		\par Conversely, let $Z(f)\cap Z(g)\neq \emptyset$ and $y\in Z(f)\cap Z(g)$. Take $h=\chi_{\{y\}}$. Then $h\in C(X)_F$ and both $hf=0$ and $hg=0$. So $h$ is adjacent to both $f$ and $g$ in the graph $\Gamma(C_c(X)_F)$.
	\end{proof} 
	
	\begin{lemma} \label{deg 3}
		For any two vertices $f,g$ there are distinct vertices $h_1$ and $h_2$ in $\Gamma(C_c(X)_F)$ such that $f$ is adjacent to $h_1$, $h_1$ is adjacent to $h_2$ and $h_2$ is adjacent to $g$ if $Z(f)\cap Z(g)= \emptyset$.
	\end{lemma}
	
	\begin{proof}
		Let $Z(f)\cap Z(g)= \emptyset$. Let us choose $x\in Z(f)$ and $y\in Z(g)$. Consider two functions $h_1=\chi_{\{x\}}$ and $h_2=\chi_{\{y\}}$. Then $Z(h_1)=X\setminus \{x\}$ and $Z(h_2)=X\setminus \{y\}.$ So $Z(h_1)\cup Z(f)=X$, $Z(h_2)\cup Z(g)=X$ and $Z(h_1)\cup Z(h_2)=X$. Hence by  Lemma \ref{deg1}, we can say that $f$ is adjacent to $h_1$, $h_1$ is adjacent to $h_2$ and $h_2$ is adjacent to $g$.
	\end{proof}
	
	\begin{definition}
		For two vertices $f,g$ in any graph $G$ , $d(f,g)$ is defined as the length of the smallest path between $f$ and $g$.
	\end{definition}

	\begin{theorem}\label{deg-thm}
		For any two vertices $f,g$ in the	graph $\Gamma(C_c(X)_F)$, we have the following outputs: 
		\begin{enumerate}
			\item[(i)] $d(f,g)=1$ if and only if $Z(f)\cup Z(g)=X$.
			\item[(ii)] $d(f,g)=2$ if and only if $Z(f)\cup Z(g)\neq X$ and $Z(f)\cap Z(g)\neq \emptyset$.
			\item[(iii)] $d(f,g)=3$ if and only if $Z(f)\cup Z(g)\neq X$ and $Z(f)\cap Z(g) = \emptyset$.
		\end{enumerate}
	\end{theorem}
	
	\begin{proof} (i) It follows from Lemma \ref{deg1}.
		\par  (ii) Let $d(f,g)=2.$ So $f$ and $g$ are not adjacent to each other. Then by Lemma \ref{deg1},  $Z(f)\cup Z(g)\neq X$. Moreover, there is a vertex $h\in \Gamma(C(X)_F)$ such that $h$ is adjacent to both $f$ and $g$. Hence by Lemma \ref{deg 2}, we have $Z(f)\cap Z(g)\neq \emptyset$.
		\par Conversely, let $Z(f)\cup Z(g)\neq X$ and $Z(f)\cap Z(g)\neq \emptyset$. Then by Lemma \ref{deg1} and \ref{deg 2}, $f$ and $g$ are not adjacent and  there is a third vertex $h$, adjacent to both $f$ and $g$. Hence $d(f,g)=2.$
		\par  (iii) Let $d(f,g)=3$. Then by Lemmas \ref{deg1} and \ref{deg 2}, we get  $Z(f)\cup Z(g)\neq X$ and $Z(f)\cap Z(g) = \emptyset$.
		\par Conversely, let $Z(f)\cup Z(g)\neq X$ and $Z(f)\cap Z(g) = \emptyset$. Then by Lemma \ref{deg1} and \ref{deg 2}, $f$ and $g$ are not adjacent to each other and there is no common vertex $h$ which is adjacent to both $f$ and $g$. Hence $d(f,g)\geq 3.$ Since $Z(f)\cap Z(g) = \emptyset$, applying Lemma \ref{deg 3}, there are two distinct vertices $h_1$ and $h_2$ such that $f$ is adjacent to $h_1$, $h_1$ is adjacent to $h_2$ and $h_2$ is adjacent to $g$. As a consequence, $d(f,g)=3.$
	\end{proof}

	\begin{definition}
		The maximum of all possible $d(f,g)$ is called the diameter of a graph $G$ and it is denoted by $diam(G)$. Also, the length of the smallest cycle in the graph $G$ is called the girth of the graph $G$ and it is denoted by $gr(G)$. If there does not exist any cycle in the graph $G$, we declare $gr (G)=\infty.$ 
	\end{definition}
	
	\begin{theorem}\label{diam}
		If a space $X$ contains at least three elements, then\\ $diam(\Gamma(C_c(X)_F))=gr(\Gamma(C_c(X)_F))=3.$
	\end{theorem}
	
	\begin{proof}
		Let us take three distinct points $x,y,z$ in $X$. Consider the functions  $f=1-\chi_{\{x\}}$ and $g=1-\chi_{\{y\}}$. Then $Z(f)=\{x\}$ and $Z(g)=\{y\}$. Thus $Z(f)\cup Z(g)\neq X$ because $z\notin Z(f)\cup Z(g).$  As $Z(f)\cap Z(g)=\emptyset$, by Theorem \ref{deg-thm}(iii), $d(f,g)=3.$ But we know that $d(f,g)\leq 3$ for all vertices $f,g$ in $\Gamma(C_c(X)_F)$. Hence we have $diam(\Gamma(C_c(X)_F))=3.$
		\par For the girth of the graph, take $h_1=\chi_{\{x\}}$, $h_2=\chi_{\{y\}}$ and $h_3=\chi_
		{\{z\}}$. Then the union of any two zero sets among $Z(h_1)$, $Z(h_2)$ and $Z(h_3)$ is $X$. Thus $h_1,h_2$ and $h_3$ form a triangle. Since there is no loop in the graph $\Gamma(C_c(X)_F)$, the girth $gr(\Gamma(C_c(X)_F))=3$.
	\end{proof}
	
	\begin{theorem}
		$diam(\Gamma(C_c(X)_F))=2$ if and only if $gr(\Gamma(C_c(X)_F))=\infty$ if and only if  $|X|=2$. 
	\end{theorem}

	\begin{proof}
		Let $X=\{x,y\}$. Then for any vertex $f$ of $\Gamma(C_c(X)_F)$,  $Z(f)$ must be singleton. Let us consider $f=\chi_{\{x\}}$ and $g=\chi_{\{y\}}$. Then $f$ and $2f$ are not adjacent to each other whereas $g$ is adjacent to both $f$ and $2f$. Now for two vertices $f$ and $g$, if their zero sets are same, then they must be constant multiple of each other and thus they cannot be adjacent and their distance is 2 and if their zero sets are not same then they are adjacent to each other. Hence for any two vertices $f,g$, $d(f,g)$ is either 1 or 2. Thus we conclude that $diam(\Gamma(C_c(X)_F))=2.$	
		\par Since there are only two distinct zero sets, there cannot exist any cycle in the graph $\Gamma(C_c(X)_F)$. Thus the girth $gr(\Gamma(C_c(X)_F))=\infty$.
		\par Now suppose $diam(\Gamma(C_c(X)_F))=2$ or girth $gr(\Gamma(C_c(X)_F))=\infty$. By Theorem \ref{diam}, we see that if $X$ contains more than two points then diameter and girth both are $3$. Hence we have $|X|=2$ because if $X$ is singleton, then there is no zero divisor.
	\end{proof}
	
	\begin{definition}
		For a vertex $f$ in a graph $G$, the associated number $e(f)$ is defined by $e(f)=\max\{d(f,g):g (\neq f)$ is a vertex in $G\}$. The vertex $g$ with smallest associated number is called a centre of the graph. The  associated number of the centre vertex in $G$ is called the radius of the graph and it is denoted by $\rho(G)$. 
	\end{definition}
	
	\par The following result is about the associated number of any vertex in the graph $\Gamma(C_c(X)_F)$.
	\begin{lemma}
		For any vertex $f$ in the graph $\Gamma(C_c(X)_F)$, we have 
		\[e(f)=\left\{
		\begin{array}{ll}
			2 \text{ if } X\setminus Z(f)\text{ is singleton}\\
			3 \text{ otherwise.}
		\end{array}
		\right.\]
	\end{lemma}
	
	\begin{proof}
		Suppose $X\setminus Z(f)=\{x_\circ\}$. Let $g$ be any vertex in $\Gamma(C_c(X)_F)$ such that $g\neq f$. Then there are only two possibilities, namely $x_\circ\in Z(g)$ or $x_\circ\notin Z(g)$. If $Z(g)$ contains $x_\circ$ then $fg=0$. In this case $f$ and $g$ are adjacent to each other. Thus $d(f,g)=1$. On the other hand, if $Z(g)$ does not contain $x_\circ$ then $Z(g)\subseteq Z(f)$. This implies that $Z(f)\cap Z(g) = Z(g)\neq \emptyset$ and $Z(f)\cup Z(g)=Z(f)\neq \emptyset$. Therefore by Theorem \ref{deg-thm}, $d(f,g)=2.$ Hence we prove that $e(f)=2.$
		\par On the other hand, let $X\setminus Z(f)$ contains at least two points, say $x_\circ$ and $y_\circ$. By Theorem \ref{deg-thm}, we see that $e(f)\leq 3.$ Now choose $g=1-\chi_{\{x_\circ\}}$. Then $Z(g)=\{x_\circ\}$. Clearly, $Z(f)\cap Z(g)=\emptyset$ and $Z(f)\cup Z(g)\neq X$ because $y_\circ$ does not belong to the union. Hence by Theorem \ref{deg-thm}, for this particular $g$, we get $d(f,g)=3$. Thus we obtain that $e(f)=3$.
	\end{proof}
	
	\begin{corollary}\label{radius}
		The radius $\rho(\Gamma(C_c(X)_F))$ of the graph $\Gamma(C_c(X)_F)$ is always 2.
	\end{corollary}
	
	\begin{proof}
		We can always consider a vertex $f$ with $e(f)=2$, for example take $f=\chi_{\{x_\circ\}}.$ Then $X\setminus Z(f)$ is singleton. So we have radius of $\Gamma(C_c(X)_F)=2$.
	\end{proof}
	
	
	\begin{definition}
		A graph $G$ is said to be
		\begin{enumerate}
			\item[(i)] triangulated if every vertex of the graph $G$ is a vertex of a triangle.
			
			\item[(ii)] hyper-triangulated if every  edge of the graph $G$ is an edge of a triangle.
		\end{enumerate}
	\end{definition}
	
	\begin{theorem}\label{try}
		The graph $\Gamma(C_c(X)_F)$ is neither triangulated nor hyper-triangulated.
	\end{theorem}
	
	\begin{proof}
		At first, we prove that the graph $\Gamma(C_c(X)_F)$ is not triangulated. For this, let us consider $x_\circ\in X.$ Now define $f= 1-\chi_{\{x_\circ\}}.$ Then $Z(f) = \{x_\circ\}.$ We claim that there is no triangle containing $f$ as a vertex. If possible, let $g,h$ be two vertices such that $f,g,h$ make a triangle. Then by Lemma \ref{deg1}, $Z(g) = X-\{x_\circ\}=Z(h).$ But again by Lemma \ref{deg1}, $g$ and $h$ cannot be adjacent. This creates a contradiction that $f,g,h$ make a triangle. 
		\par Now to prove hypertriangulated,  let us take a point $x_\circ\in X$. Now take two functions $f=\chi_{\{x_\circ\}}$ and $g=1-\chi_{\{x_\circ\}}$. Then $Z(f)\cup Z(g)=X$ and $Z(f)\cap Z(g)=\emptyset$. Then by Lemma \ref{deg 2}, it is not possible to get a triangle that contains the edge connecting $f$ and $g$. So the graph $\Gamma(C_c(X)_F)$ is neither triangulated nor hyper-triangulated.
	\end{proof}
	
	The above mentioned  result is totally different from the case of $C(X)$. In fact, we have 
	
	\begin{proposition}[\cite{zero}]
		The following results are true.
		\begin{enumerate}
			\item [(i)]$\Gamma(C(X))$ is triangulated if and only if $X$ does not contain any non-isolated points. 
			\item[(ii)] $\Gamma(C(X))$ is hyper-triangulated if and only if $X$ is a connected middle $P$-space.
		\end{enumerate}
	\end{proposition}
	For definition of middle $P$-space see \cite{zero}.
	
	\begin{definition}
		For two vertices $f$ and $g$ in any graph $G$, we denote by $c(f,g)$ the length of the smallest cycle containing $f$ and $g$. If there is no cycle containing $f$ and $g$, we declare $c(f,g)=\infty$.
	\end{definition}
	
	\par In the following theorem, we shall discuss all possible values of  $c(f,g)$ in the graph $\Gamma(C_c(X)_F)$.
	
	\begin{theorem}
		Let $f$ and $g$ be two vertices in the graph $\Gamma(C_c(X)_F)$. Then
		\begin{enumerate}
			\item[(i)] $c(f,g)=3$ if and only if $Z(f)\cup Z(g)=X$ and $Z(f) \cap Z(g)\neq \emptyset$.
			\item[(ii)] $c(f,g)=4$ if and only if $Z(f)\cup Z(g)=X$ and $Z(f) \cap Z(g)= \emptyset$ or $Z(f)\cup Z(g)\neq X$ and $Z(f) \cap Z(g) \neq \emptyset$.
			\item[(iii)] $c(f,g)=6$ if and only if $Z(f)\cup Z(g)\neq X$ and $Z(f) \cap Z(g) = \emptyset$.
		\end{enumerate}
	\end{theorem} 
	
	\begin{proof}
		(i) Suppose $Z(f)\cup Z(g)=X$ and $Z(f) \cap Z(g)\neq \emptyset$. Thus by Lemma \ref{deg1} and \ref{deg 2}, $f$ and $g$ are adjacent to each other and there is another vertex $h$ adjacent to both $f$ and $g$. Hence we obtain a triangle with vertices $f,g$ and $h$. This shows that $c(f,g)=3$. Conversely, if $c(f,g)=3$ then there exists a triangle with $f,g$ and $h$ as its vertices for some other vertex $h$. Now using Lemma \ref{deg1} and \ref{deg 2}, we find that $Z(f)\cup Z(g)=X$ and $Z(f) \cap Z(g)\neq \emptyset$.\\
		
		(ii) Consider $Z(f)\cup Z(g)=X$ and $Z(f) \cap Z(g)= \emptyset$. Then using Lemma \ref{deg1}, $f$ and $g$ are adjacent to each other. Now by Lemma \ref{deg 3}, there are vertices $h_1$ and $h_2$ such that $f$ is adjacent to $h_1$, $h_1$ is adjacent to  $h_2$ and $h_2$ is adjacent to $g$. Thus we get a cycle of length 4 with vertices in order, $f,h_1,h_2$ and $ g$. As $Z(f) \cap Z(g)= \emptyset$, by Lemma \ref{deg 2}, there is no triangle containing $f$ and $g$ as its vertices. Thus $c(f,g)=4$.
		
		\par Now suppose $Z(f)\cup Z(g)\neq X$ and $Z(f) \cap Z(g) \neq \emptyset$. Then using Lemma \ref{deg1}, $f$ and $g$ are not adjacent to each other. By Lemma \ref{deg 2}, there exists a vertex $h$ such that $h$ is adjacent to both $f$ and $g$. Then $2h$ is also adjacent to both $f$ and $g$. Thus we get a quadrilateral containing vertices in order $f,h,g$ and $2h$. Again condition $Z(f)\cup Z(g)\neq X$ implies that it is not possible to have a triangle containing $f$ and $g$ as its vertices. So $c(f,g)=4.$
		\par To prove the converse, let $c(f,g)=4.$ Now $Z(f)\cup Z(g)=X$, then we must have  $Z(f)\cap Z(g)=\emptyset$, otherwise we have a triangle having vertices $f$ and $g$. If we have $Z(f)\cup Z(g)\neq X$, then $f$ and $g$ are not adjacent to each other. But there is a quadrilateral containing $f$ and $g$. So there must exist two functions $h_1$ and $h_2$ such that both $h_1$ and $h_2$  are adjacent to both $f$ and $g$. So by Lemma \ref{deg 2}, we have $Z(f) \cap Z(g) \neq \emptyset$.
		
		(iii) Let $Z(f)\cup Z(g)\neq X$ and $Z(f) \cap Z(g) = \emptyset$. Then $f$ and $g$ are not adjacent to each other. As $Z(f) \cap Z(g) = \emptyset$, by Lemma \ref{deg 3}, there are two vertices $h_1$ and $h_2$ in $\Gamma (C(X)_F)$ such that there is a path connecting $f,h_1,h_2 $ and $g$ in order. So immediately there is another path connecting $g,2h_2,2h_1$ and $f$. So we get a cycle of length 6, namely $f,h_1,h_2,g,2h_2,2h_1$ and $f$. Let us make it clear that with the given condition it is not possible to get a cycle of length 5. As $f$ and $g$ are not adjacent to each other, to have a cycle of length 5, we must have a path of length 2 joining $f$ and  $g$ which is not possible as $Z(f)\cap Z(g)=\emptyset.$ This implies that $c(f,g)=6.$
		\par Conversely, let $c(f,g)=6.$ Then by proof of (i) and (ii), we have $Z(f)\cup Z(g)\neq X$ and $Z(f) \cap Z(g) = \emptyset$. \end{proof}

\end{document}